\documentclass{article}
\usepackage{graphicx} 
\usepackage[utf8]{inputenc}
\usepackage{amsmath}
\usepackage{amsfonts}
\usepackage{amssymb,dsfont}
\usepackage{amsthm}
\usepackage{comment}
\usepackage{diagbox}
\usepackage{hyperref}
\usepackage[margin=3.6cm]{geometry}
\usepackage[bottom]{footmisc}

\newtheorem{theorem}{Theorem}[section]

\newtheorem{lemma}[theorem]{Lemma}
\newtheorem{proposition}[theorem]{Proposition}
\theoremstyle{definition}

\theoremstyle{remark}
\newtheorem{remark}{Remark}[section]

\theoremstyle{definition}

\title{A Moser-spindle-free 5-chromatic unit distance graph on 2131 vertices in the plane}
\author{Jan Kristian Haugland\footnote{E-mail address: \href{mailto:admin@neutreeko.net}{\texttt{admin@neutreeko.net}}}}
\date{}
\begin{document}

\maketitle

\section{Introduction}
The Hadwiger-Nelson problem asks for the minimum number of colours $\chi$ required to colour the points of the Euclidean plane such that any two points at a unit distance apart receive different colours. It dates back to 1950, and for almost seven decades, the best bounds were $4 \leq \chi \leq 7$ \cite{soifer}. The lower bound comes from a 4-chromatic unit distance graph on 7 vertices, known as the Moser spindle; confer Fig. \ref{moser}. The upper bound comes from a tesselation of the plane by regular hexagons with 7 colours such that the diameter of each hexagon is slightly less than 1, while the distance between distinct hexagons of the same colour is slightly greater than 1. In 2018, de Grey \cite{grey} found a 5-chromatic unit distance graph, thus proving that $\chi \geq 5$. The first example had 20425 vertices, which was soon improved to 1581 vertices. During the following few years, smaller examples were found, especially by Heule \cite{heule1, heule2, heule3}, and the current record is 509 vertices by Parts \cite{parts}.

\begin{figure}
\centering
\includegraphics[width=0.45\linewidth]{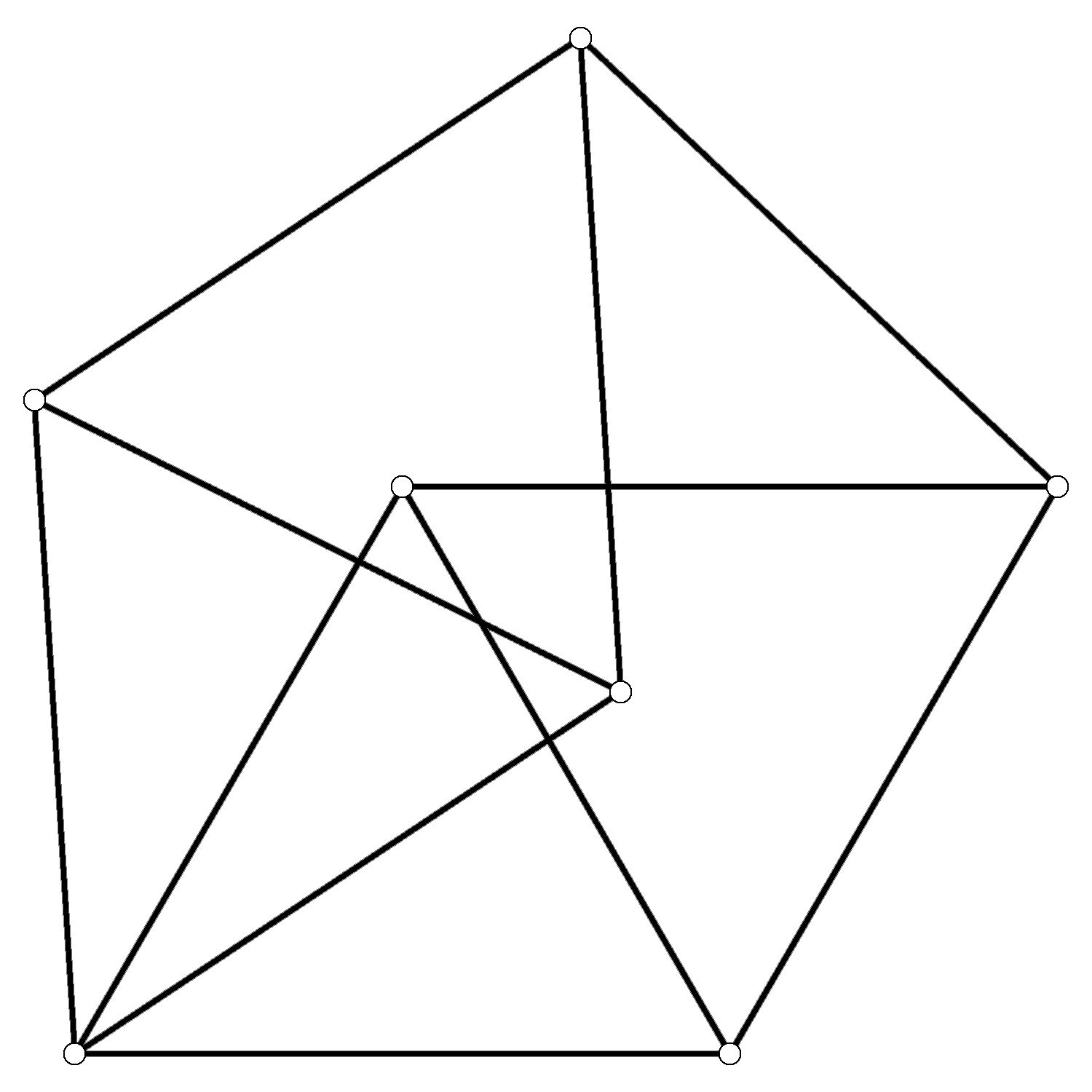}
\caption{The Moser spindle}
\label{moser}
\end{figure}

All known small 5-chromatic unit distance graphs rely heavily on the Moser spindle, and a natural question is whether this is a structural necessity for non-4-colourability. Voronov \textit{et al.} \cite{voronov} answered this in the negative by constructing 5-chromatic unit distance graphs on 64513 vertices entirely free of the Moser spindle, and Heule \cite{heule4} subsequently reduced this bound to 1441 vertices. Rather than focusing purely on minimisation, the purpose of this note is to introduce an entirely different geometric framework for such constructions.

In Section 2, we present a 7-fold symmetric unit distance graph $H$ on 21 vertices. In Section 3, the arcs of $H$ are utilized to construct a graph $G_1$ on 740 vertices for which $(0, 0)$ and $(0, \sqrt{3})$ have different colours in any 4-colouring. Combining four isometric copies of $G_1$ yields a Moser-spindle-free 5-chromatic unit distance graph on 2131 vertices. In Section 4, we give an upper bound of 6 for the number of 4-colourings of the lattice $L$ generated by the arcs of $H$, up to equivalence.

\section{A unit distance graph based on the regular heptagon}

Let $$\alpha = \frac{1}{\operatorname{sin} \frac{2\pi}{7}}, \, \beta = \frac{1}{\operatorname{sin} \frac{4\pi}{7}}, \, \gamma = \frac{1}{\operatorname{sin} \frac{8\pi}{7}}.$$ For $0 \leq j < 7$, let $$P_j=-\frac{\gamma}{2} \left( \operatorname{cos}\left( \frac{2 \pi j}{7} + \frac{\pi}{2} \right), \, \operatorname{sin}\left( \frac{2 \pi j}{7} + \frac{\pi}{2} \right) \right),$$ $$Q_j=\frac{\alpha}{2} \left( \operatorname{cos}\left( \frac{2 \pi j}{7} + \frac{\pi}{6} \right), \, \operatorname{sin}\left( \frac{2 \pi j}{7} + \frac{\pi}{6} \right) \right),$$ $$R_j=\frac{\beta}{2} \left( \operatorname{cos}\left( \frac{2 \pi j}{7} + \frac{5\pi}{6} \right), \, \operatorname{sin}\left( \frac{2 \pi j}{7} + \frac{5\pi}{6} \right) \right),$$ and let $H$ be the unit distance graph induced by $\{P_j\}_{0 \leq j < 7} \cup \{Q_j\}_{0 \leq j < 7} \cup \{R_j\}_{0 \leq j < 7}$.

\begin{proposition}
The graph $H$ consists of a regular heptagon, two regular heptagrams and seven equilateral triangles, as shown in Fig. \ref{heptagon}.
\end{proposition}

\begin{figure}
\centering
\includegraphics[width=0.64\linewidth]{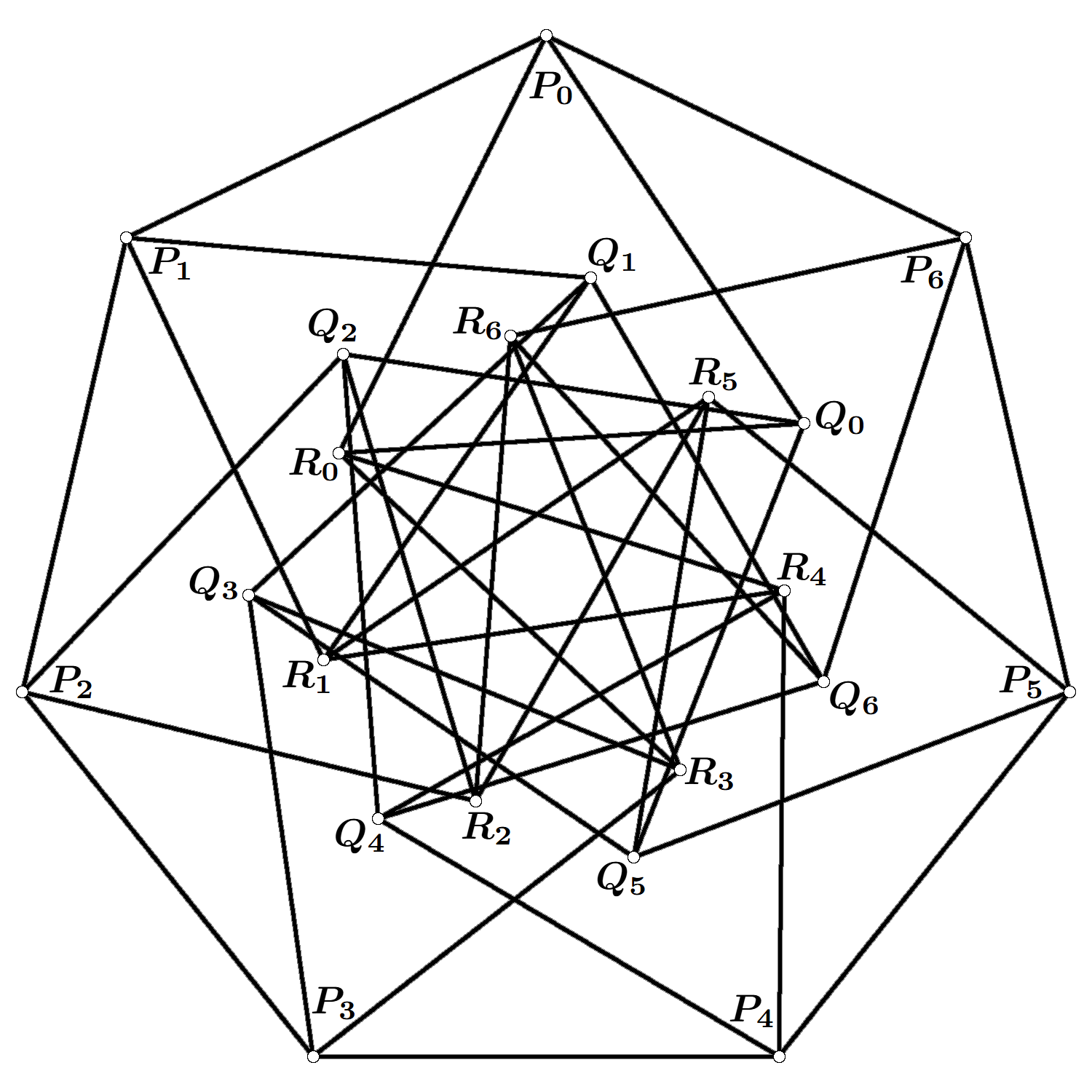}
\caption{The unit distance graph $H$}
\label{heptagon}
\end{figure}

\begin{proof}
By basic trigonometry, $\{P_j\}_{0 \leq j < 7}$ are the vertices of a regular heptagon, $\{Q_j\}_{0 \leq j < 7}$ are the vertices of a regular heptagram $\{7/2\}$ and $\{R_j\}_{0 \leq j < 7}$ are the vertices of a regular heptagram $\{7/3\}$, all of unit side length, and it remains to verify that $P_jQ_jR_j$ is an equilateral triangle of unit side length for each $j$. To this end, we establish two identities involving $\alpha$, $\beta$ and $\gamma$.

First, we have $$\operatorname{tan} \frac{\theta}{2} = \frac{\operatorname{sin}\theta}{1+\operatorname{cos}\theta}$$ which yields $$\frac{1}{\operatorname{sin}\theta}=\frac{1}{\operatorname{tan} \frac{\theta}{2}}-\frac{1}{\operatorname{tan}\theta}$$ and therefore
\begin{equation}\label{dim1}
    \alpha+\beta+\gamma=\frac{1}{\operatorname{tan}\frac{\pi}{7}}-\frac{1}{\operatorname{tan}\frac{2\pi}{7}}+\frac{1}{\operatorname{tan}\frac{2\pi}{7}}-\frac{1}{\operatorname{tan}\frac{4\pi}{7}}+\frac{1}{\operatorname{tan}\frac{4\pi}{7}}-\frac{1}{\operatorname{tan}\frac{8\pi}{7}}=0.
\end{equation}

Second, we have in general $$\sum_{j=1}^{n-1}\frac{1}{\operatorname{sin}^2 \frac{j\pi}{n}}=\frac{n^2-1}{3}$$ (for example, confer \cite{fisher}, identity (23)). Taking $n=7$, it follows that
\begin{equation}\label{dim2}
    \alpha^2+\beta^2+\gamma^2=\frac{1}{2}\left(\frac{1}{\operatorname{sin}^2 \frac{\pi}{7}} + \frac{1}{\operatorname{sin}^2 \frac{2\pi}{7}} + \dotsc + \frac{1}{\operatorname{sin}^2 \frac{6\pi}{7}}\right) = \frac{7^2-1}{2 \cdot 3}=8.
\end{equation}

\noindent In view of \eqref{dim1} and \eqref{dim2}, we have
\begin{equation*}
\begin{cases}
    \alpha^2 + \alpha \beta + \beta^2 = \frac{\alpha^2+\beta^2+(\alpha+\beta)^2}{2}=\frac{\alpha^2+\beta^2+(-\gamma)^2}{2}=4\\
    \alpha^2 + \alpha \gamma + \gamma^2 = \frac{\alpha^2+\gamma^2+(\alpha+\gamma)^2}{2}=\frac{\alpha^2+\gamma^2+(-\beta)^2}{2}=4\\
    \beta^2 + \beta \gamma + \gamma^2 = \frac{\beta^2+\gamma^2+(\beta+\gamma)^2}{2}=\frac{\beta^2+\gamma^2+(-\alpha)^2}{2}=4
\end{cases}
\end{equation*}
which is equivalent to
\begin{equation*}
\begin{cases}
    \left( \frac{\alpha}{2}\right)^2+\left( \frac{\beta}{2}\right)^2-\frac{\alpha \beta}{2}\operatorname{cos}\frac{2\pi}{3}=1\\
    \left( \frac{\alpha}{2}\right)^2+\left( -\frac{\gamma}{2}\right)^2-\frac{\alpha (-\gamma)}{2}\operatorname{cos}\frac{\pi}{3}=1\\
    \left( \frac{\beta}{2}\right)^2+\left( -\frac{\gamma}{2}\right)^2-\frac{\beta (-\gamma)}{2}\operatorname{cos}\frac{\pi}{3}=1.
\end{cases}
\end{equation*}
By the law of cosines, with $(0, 0)$ as the vertex opposite either side of the triangle, each side length must indeed be equal to 1.
\end{proof}

To construct a 5-chromatic unit distance graph, we utilize the 84 unit vectors occurring as arcs in $H$ when viewed as a symmetric directed graph, sorted by their angles with the positive $x$-axis. The edges of the regular heptagon and the regular heptagrams cover all angles that are integral multiples of $\frac{\pi}{21}$. If we multiply the vector $\overrightarrow{Q_0P_0}$ by $4 \operatorname{sin} \frac{2\pi}{7}=8\operatorname{sin}\frac{\pi}{7} \operatorname{cos}\frac{\pi}{7}$, we get $(-\sqrt{3}, 4 \operatorname{cos}\frac{\pi}{7}-1)$, so that the angle is $\varphi = \pi - \operatorname{tan}^{-1} \frac{4 \operatorname{cos} \frac{\pi}{7} - 1}{\sqrt{3}}$, which lies between $14 \cdot \frac{\pi}{21}$ and $15 \cdot \frac{\pi}{21}$. Accordingly, we put $\theta = \frac{21}{\pi} \, \varphi - 14 \approx 0.42363201413287$ and define $$u_{2j}=\left(\operatorname{cos} \frac{\pi j}{21}, \, \operatorname{sin} \frac{\pi j}{21} \right), \, \, u_{2j+1}=\left(\operatorname{cos} \frac{\pi (j + \theta)}{21}, \, \operatorname{sin} \frac{\pi (j + \theta)}{21} \right)$$ for $0 \leq j < 42$. The corresponding arcs in $H$ are listed in Table \ref{unitvectors}, where the indices of the vertices are taken modulo 7. The lattice generated by $\{u_j\}_{0 \leq j < 84}$ is called $L$.

\begin{table}
\centering
\begin{tabular}{|c|c|c|}\hline
Vector & Tail & Head \\ \hline
$u_{12j}$ & $P_{3+j}$ & $P_{4+j}$ \\
$u_{1+12j}$ & $R_j$ & $Q_j$ \\
$u_{2+12j}$ & $R_{1+j}$ & $R_{4+j}$ \\
$u_{3+12j}$ & $R_{6+j}$ & $P_{6+j}$ \\
$u_{4+12j}$ & $Q_{4+j}$ & $Q_{6+j}$ \\
$u_{5+12j}$ & $Q_{5+j}$ & $P_{5+j}$ \\
$u_{6+12j}$ & $P_{1+j}$ & $P_j$ \\
$u_{7+12j}$ & $Q_{4+j}$ & $R_{4+j}$ \\
$u_{8+12j}$ & $R_{1+j}$ & $R_{5+j}$ \\
$u_{9+12j}$ & $P_{3+j}$ & $R_{3+j}$ \\
$u_{10+12j}$ & $Q_{3+j}$ & $Q_{1+j}$ \\
$u_{11+12j}$ & $P_{2+j}$ & $Q_{2+j}$ \\ \hline
\end{tabular}
\caption{Labelling of unit vectors occurring as arcs in $H$ ($0\leq j < 7$)}
\label{unitvectors}
\end{table}

\begin{lemma}\label{distance}
Suppose $S = \{s_1, \dotsc, s_n\}$ where $2 \leq n \leq 8$ and $0 \leq s < 84 \, \forall \, s \in S$, and let $(x, y) = \sum_{s \in S} u_s$. Then we either have $|x| > \varepsilon_n$, or $|y| > \varepsilon_n$, or $x=y=0$, where $\varepsilon_2=\varepsilon_3=0.045$, $\varepsilon_4=\varepsilon_5=0.0067$, $\varepsilon_6=0.0016$, $\varepsilon_7=0.00089$ and $\varepsilon_8=0.00071$.
\end{lemma}

\begin{proof}
It has been verified by an exhaustive computer search using standard double-precision floats that for each $2 \leq n \leq 8$, if $|x| < \varepsilon_n$ and $|y| < \varepsilon_n$, then $\{u_s\}_{s \in S}$ form a combination of closed circuits, where each circuit lies in either $H$ or its image under an arc-preserving isometry. Thus,  $(x, y) = (0, 0)$.
\end{proof}

\begin{remark}\label{indices}
An example of a closed circuit in $H$ is given by $P_3P_4Q_4Q_2P_2$, which implies that $u_0+u_{35}+u_{22}+u_{53}+u_{72}=0$. An arc-preserving rotation of $H$ corresponds to an even shift of the unit vector indices (modulo 84), while an arc-preserving reflection of $H$ corresponds to a fixed-point-free reflection of the indices.
\end{remark}

\section{Construction of a 5-chromatic graph}
\textbf{Definitions and notation:} Following the terminology in \cite{parts}, a monochromatic pair of vertices are two vertices that have the same colour in any $k$-colouring for some given $k$, and a non-monochromatic pair consists of two vertices that cannot have the same colour. A point $A$ in the plane and a sequence $i_1 \, i_2 \, \dotsc \, i_n$ of indices represent a polygonal path via the vertices $A, \, A+u_{i_1}, \, A + u_{i_1} + u_{i_2}, \, \dotsc, \, A + u_{i_1} + \dotsc + u_{i_n}$. The vertex set of a graph $G$ is denoted by $V(G)$. The $m$-core of $G$ is the subgraph obtained by iteratively deleting vertices of degree less than $m$ until all remaining vertices have degree at least $m$. The $r$-ball of $G$ centred at a vertex $v \in V(G)$ is the set of all vertices at a graph distance of at most $r$ from $v$.
\bigskip

We start by constructing a graph in the plane for which $A=(0, 0)$ and $B=(0, \sqrt{3})$ form a non-monochromatic pair for $k=4$. Intuitively, vertices that have a shorter total graph distance to $A$ and $B$ tend to be more significant for this purpose. Thus for $n \in \{5, 6\}$, let $T_n$ be the unit distance graph induced by all vertices on some polygonal path from $A$ to $B$ of graph length $\leq n$, such that each arc is in $\{u_j\}$.

By Lemma \ref{distance}, since there is a path from $B$ back to $A$ using the vectors $u_{56}$ and $u_{70}$, the endpoint of the polynomial path represented by $A$ and $i_1 \, \dotsc \, i_n$ with $n \leq 6$ is equal to $B$ provided $\sum_{k \leq n} u_{i_k} = (x, y)$ with $|x| < \varepsilon_8$ and $|y-\sqrt{3}| < \varepsilon_8$. Hence, we can find the individual valid paths numerically.

Moreover, if $(x, y), \, (x', y')$ are two vertices of $T_6$, their mutual graph distance is at most 6, either via $A$ or via $B$. It follows that they are identical if and only if $|x-x'|<\varepsilon_6$ and $|y-y'|<\varepsilon_6$, and they can therefore be distinguished numerically.

Thus, it can be verified that $|V(T_5)| = 1042$ and $|V(T_6)| = 12856$. Unfortunately, $(A, B)$ is not a non-monochromatic pair in $T_5$, whereas $T_6$ is rather large. However, we can navigate between these two undesirable cases by defining a graph $G_0$ induced by $$\{v \in V(T_6) \, | \, v \in V(T_5) \text{ or } v \text{ is adjacent to at least 7 vertices in } V(T_5) \}$$ and letting $G_1$ be the 7-core of $G_0$. The latter graph consists of 740 vertices and 3985 edges. With $A$ and $B$ required to have the same colour, $G_1$ cannot be 4-coloured, as verified by the SAT-solver CaDiCaL, and it follows that $(A, B)$ is a non-monochromatic pair. Curiously, the unsatisfiable core retained variables corresponding to all vertices of this graph.

A set of polygonal paths from $A$ to $B$ covering all vertices of $G_1$ is given in Appendix A.

Let $V_1$, $V_2$ be isometric copies of $V(G_1)$ obtained by rotating $V(G_1)$ by an angle $\frac{\pi}{3}$ clockwise around $\left(-\frac{1}{2}, \frac{\sqrt{3}}{2}\right)$ and by an angle $\frac{\pi}{3}$ counterclockwise around $\left(\frac{1}{2}, \frac{\sqrt{3}}{2}\right)$, respectively. I.e., $$V_1 = \left\{\frac{1}{2}\, (x, \, y) \left( \begin{matrix}1 & -\sqrt{3} \\ \sqrt{3} & 1\end{matrix} \right) + (-1, 0) \, \middle| \, (x, \, y) \in V(G_1) \right\}_{}$$ and $$V_2 = \left\{ \frac{1}{2} \, (x, \, y) \left( \begin{matrix}1 & \sqrt{3} \\ -\sqrt{3} & 1\end{matrix} \right) + (1, 0) \, \middle| \, (x, \, y) \in V(G_1) \right\}.$$ The graph $G_2$ induced by $V_1 \cup V_2$ consists of 1066 vertices and 6264 edges. Among the vertices $(-1, 0)$, $(0, 0)$, $(1, 0)$, $\left(-\frac{1}{2}, \frac{\sqrt{3}}{2}\right)$ and $\left(\frac{1}{2}, \frac{\sqrt{3}}{2}\right)$, the only two that do not form a non-monochromatic pair are $(-1, 0)$ and $(1, 0)$, which accordingly must form a monochromatic pair. Finally, consider the spindle $G_3$ induced by $$\left\{ (x, \, y), \, \frac{1}{8} \, (x+1, \, y) \left( \begin{matrix}7 & \sqrt{15} \\ -\sqrt{15} & 7\end{matrix} \right) + (-1, 0) \, \middle| \, (x, \, y) \in V(G_2) \right\},$$ which is a graph on $2\times 1066-1=2131$ vertices and $2 \times 6264+2=12530$ edges. It is not 4-colourable, since $(1, 0)$ and $\left( \frac{3}{4}, \frac{\sqrt{15}}{4} \right)$ are adjacent, and each of them is in a monochromatic pair with $(-1, 0)$. On the other hand, a SAT-solver can quickly find a 5-colouring of $G_3$, thus establishing that it is 5-chromatic. It has also been verified that it does not contain the Moser spindle as a subgraph.

\section{4-colourings of \textit{L}}
In this section, we outline an argument showing that the lattice $L$ has at most 6 distinct 4-colourings, up to isometry and colour permutations. Given a 4-colouring of an induced subgraph $G$ of $L$, we label each edge of $G$ with $A$, $B$ or $C$ based on its endpoint colours:\begin{itemize}
    \item Label $A$: Endpoints have colours $\{0, \, 1\}$ or $\{2, \, 3\}$.
    \item Label $B$: Endpoints have colours $\{0, \, 2\}$ or $\{1, \, 3\}$.
    \item Label $C$: Endpoints have colours $\{0, \, 3\}$ or $\{1, \, 2\}$.
\end{itemize}

Let $G_4$ be the 3-ball of $L$ centred at $(0, \, 0)$, which contains 83581 vertices. A count of all 4-colourings of $G_4$ was performed using the SAT solver Glucose, with virtual edges added to force distinct colours on vertices at distance $\sqrt{3}$.

The solver results reveal strict structural regularities. For any 4-colouring, any two parallel edges of length 1 receive the same label. Consequently, the edge-labelling of $G_4$ is uniquely determined by the labels assigned to the unit vectors $\{u_j\}_{0 \leq j < 84}$. Under the index convention of Remark \ref{indices}, antiparallel vectors $u_j$ and $u_{j+42}$ receive the same label.

Furthermore, the label of $u_{j+14}$ shifts uniformly from that of $u_j$ for each $j$, following either the cycle $A \to B \to C \to A$ or $A \to C \to B \to A$. Thus, the labelling of the unit vectors $\{u_j\}_{0 \leq j < 84}$ is determined by the labels assigned to the first 14 vectors and the universal shift. Up to isometry and label permutations, the feasible labellings found by the solver appear in Table \ref{edgelabels}.

We have not proved that these local colourings extend to the entire lattice $L$. However, we deduce that any global 4-colouring must match the unique extension of a $G_4$ colouring that preserves these label properties.

\begin{table}
\centering
\begin{tabular}{|c|cccccccccccccc|}\hline
Labelling&$u_0$&$u_1$&$u_2$&$u_3$&$u_4$&$u_5$&$u_6$&$u_7$&$u_8$&$u_9$&$u_{10}$&$u_{11}$&$u_{12}$&$u_{13}$\\ \hline
1&$A$&$A$&$A$&$A$&$A$&$C$&$A$&$A$&$B$&$A$&$B$&$B$&$C$&$B$\\
2&$A$&$A$&$A$&$A$&$C$&$A$&$A$&$B$&$C$&$C$&$A$&$B$&$B$&$C$\\
3&$A$&$A$&$A$&$B$&$B$&$C$&$B$&$B$&$A$&$A$&$C$&$A$&$A$&$B$\\
4&$A$&$A$&$A$&$C$&$B$&$B$&$A$&$C$&$B$&$B$&$A$&$C$&$C$&$C$\\
5&$A$&$B$&$A$&$B$&$A$&$B$&$B$&$A$&$A$&$C$&$A$&$C$&$C$&$B$\\
6&$A$&$B$&$A$&$C$&$A$&$C$&$C$&$B$&$C$&$B$&$A$&$B$&$C$&$A$\\ \hline
\end{tabular}
\caption{Labellings of the unit vectors $u_0$ to $u_{13}$, assuming the universal shift $A \to B \to C \to A$}
\label{edgelabels}
\end{table}

\section*{Note added in revision}
Following the posting of version 4 of this preprint, Wong \cite{wong} has reduced the size of $G_3$ to 1501 vertices, and Project Threefold \cite{threefold} has shown that the proposed colourings of Section 4 can indeed be extended to the full lattice. I thank them for building upon my work.

\section*{Declaration on AI use}
Google's Gemini was used for installing and running SAT solvers, and for some language editing. All mathematical content is the author's own intellectual work.

\newpage
\section*{Appendix A: The vertices of \textit{G}\textsubscript{1}}
Below is a set of polygonal paths from $(0, 0)$ to $(0, \sqrt{3})$ that collectively visit all vertices of $G_1$, represented by their unit vector indices. It was found using the Python library SetCoverPy.\newline \newline

\footnotesize{\begin{tabular}{|ccccc|}
\hline
2 & 23 & 28 & 69 & 33 \\
7 & 19 & 28 & 69 & 38 \\
10 & 31 & 68 & 37 & 14 \\
33 & 28 & 20 & 66 & 3 \\
37 & 28 & 25 & 74 & 3 \\
40 & 17 & 53 & 7 & 14 \\
42 & 23 & 14 & 65 & 14 \\
44 & 7 & 28 & 78 & 25 \\
46 & 7 & 14 & 44 & 13 \\
47 & 83 & 37 & 14 & 16 \\
48 & 28 & 0 & 6 & 28 \\
50 & 19 & 17 & 77 & 28 \\
50 & 28 & 17 & 77 & 19 \\
53 & 32 & 83 & 14 & 22 \\
53 & 40 & 7 & 14 & 17 \\
54 & 17 & 28 & 80 & 23 \\
62 & 83 & 28 & 25 & 23 \\
64 & 7 & 33 & 31 & 14 \\
66 & 3 & 20 & 33 & 28 \\
69 & 19 & 28 & 38 & 7 \\
69 & 33 & 28 & 2 & 23 \\
69 & 34 & 17 & 7 & 33 \\
69 & 38 & 19 & 7 & 28 \\
72 & 39 & 14 & 33 & 12 \\
77 & 10 & 17 & 39 & 37 \\
77 & 28 & 17 & 50 & 19 \\
77 & 46 & 28 & 10 & 23 \\
78 & 7 & 44 & 28 & 25 \\
80 & 28 & 17 & 54 & 23 \\
82 & 19 & 58 & 28 & 25 \\
82 & 28 & 26 & 54 & 14 \\
83 & 23 & 28 & 62 & 25 \\
83 & 36 & 48 & 14 & 17 \\ \hline
\end{tabular}}
\quad
\footnotesize{\begin{tabular}{|cccccc|}
\hline
0 & 21 & 83 & 33 & 28 & 52 \\
2 & 14 & 47 & 30 & 83 & 37 \\
2 & 33 & 14 & 12 & 58 & 39 \\
2 & 47 & 30 & 37 & 14 & 83 \\
3 & 20 & 37 & 33 & 60 & 7 \\
3 & 25 & 4 & 28 & 60 & 37 \\
3 & 34 & 14 & 31 & 47 & 80 \\
3 & 36 & 8 & 69 & 34 & 28 \\
3 & 37 & 13 & 69 & 39 & 22 \\
3 & 39 & 28 & 66 & 29 & 8 \\
4 & 14 & 49 & 28 & 46 & 7 \\
4 & 17 & 25 & 28 & 69 & 48 \\
4 & 25 & 51 & 14 & 53 & 17 \\
4 & 39 & 28 & 32 & 69 & 8 \\
6 & 14 & 37 & 68 & 42 & 19 \\
6 & 19 & 28 & 7 & 57 & 44 \\
6 & 28 & 57 & 44 & 7 & 19 \\
7 & 14 & 78 & 50 & 33 & 31 \\
7 & 17 & 44 & 5 & 54 & 28 \\
7 & 28 & 64 & 8 & 44 & 25 \\
7 & 31 & 69 & 3 & 34 & 33 \\
7 & 33 & 14 & 78 & 31 & 50 \\
8 & 14 & 53 & 25 & 58 & 21 \\
8 & 21 & 0 & 58 & 39 & 28 \\
8 & 28 & 73 & 39 & 37 & 6 \\
8 & 29 & 66 & 28 & 39 & 3 \\
10 & 17 & 23 & 77 & 51 & 39 \\
10 & 23 & 14 & 51 & 68 & 31 \\
10 & 28 & 37 & 68 & 0 & 31 \\
10 & 37 & 60 & 7 & 14 & 38 \\
12 & 25 & 68 & 19 & 28 & 58 \\
13 & 0 & 39 & 37 & 28 & 76 \\
13 & 14 & 51 & 30 & 61 & 14 \\ \hline
\end{tabular}}
\quad
\footnotesize{\begin{tabular}{|cccccc|}
\hline
14 & 17 & 7 & 54 & 53 & 26 \\
14 & 22 & 69 & 42 & 3 & 34 \\
14 & 59 & 82 & 28 & 40 & 17 \\
14 & 79 & 18 & 37 & 28 & 60 \\
16 & 69 & 47 & 14 & 37 & 13 \\
17 & 19 & 36 & 53 & 14 & 72 \\
18 & 69 & 28 & 33 & 71 & 20 \\
19 & 52 & 21 & 14 & 47 & 83 \\
20 & 14 & 47 & 19 & 51 & 82 \\
20 & 71 & 33 & 28 & 69 & 18 \\
21 & 14 & 38 & 51 & 14 & 74 \\
22 & 69 & 14 & 13 & 53 & 32 \\
23 & 60 & 32 & 33 & 83 & 7 \\
24 & 7 & 28 & 52 & 19 & 69 \\
25 & 23 & 76 & 48 & 83 & 28 \\
25 & 83 & 47 & 19 & 82 & 37 \\
26 & 14 & 39 & 72 & 33 & 82 \\
28 & 9 & 28 & 66 & 3 & 40 \\
28 & 16 & 83 & 0 & 47 & 37 \\
28 & 17 & 40 & 7 & 53 & 0 \\
28 & 18 & 71 & 69 & 33 & 20 \\
28 & 24 & 75 & 14 & 66 & 33 \\
28 & 25 & 12 & 77 & 64 & 33 \\
28 & 63 & 18 & 16 & 69 & 28 \\
28 & 73 & 63 & 12 & 28 & 25 \\
30 & 77 & 28 & 17 & 43 & 80 \\
31 & 0 & 33 & 28 & 7 & 64 \\
31 & 14 & 48 & 24 & 69 & 7 \\
32 & 1 & 54 & 83 & 28 & 23 \\
32 & 14 & 13 & 69 & 22 & 53 \\
33 & 12 & 83 & 32 & 77 & 39 \\
33 & 14 & 66 & 38 & 83 & 21 \\
33 & 17 & 66 & 18 & 83 & 39 \\ \hline
\end{tabular}}

\footnotesize{\begin{tabular}{|cccccc|}
\hline
33 & 20 & 7 & 67 & 46 & 14 \\
33 & 21 & 14 & 31 & 64 & 77 \\
33 & 23 & 21 & 77 & 46 & 83 \\
33 & 23 & 80 & 17 & 72 & 39 \\
33 & 28 & 82 & 11 & 69 & 34 \\
33 & 31 & 50 & 14 & 7 & 78 \\
33 & 37 & 83 & 68 & 25 & 12 \\
34 & 17 & 15 & 51 & 14 & 68 \\
34 & 28 & 17 & 80 & 47 & 0 \\
34 & 33 & 3 & 6 & 28 & 66 \\
36 & 17 & 28 & 0 & 83 & 48 \\
37 & 6 & 33 & 68 & 5 & 28 \\
37 & 13 & 14 & 60 & 3 & 36 \\
37 & 14 & 74 & 3 & 42 & 25 \\
37 & 17 & 54 & 9 & 80 & 28 \\
37 & 24 & 39 & 80 & 77 & 17 \\
37 & 25 & 68 & 15 & 78 & 28 \\
37 & 28 & 60 & 4 & 3 & 25 \\
37 & 83 & 14 & 30 & 47 & 2 \\
38 & 7 & 14 & 60 & 37 & 10 \\
38 & 21 & 66 & 14 & 83 & 33 \\
39 & 7 & 78 & 6 & 37 & 33 \\
39 & 12 & 28 & 54 & 81 & 14 \\
39 & 13 & 14 & 51 & 76 & 23 \\
39 & 14 & 18 & 74 & 17 & 51 \\
39 & 17 & 14 & 18 & 74 & 51 \\
39 & 18 & 14 & 8 & 69 & 42 \\
39 & 22 & 69 & 13 & 3 & 37 \\
39 & 37 & 14 & 6 & 62 & 13 \\
40 & 3 & 53 & 14 & 31 & 7 \\
40 & 7 & 3 & 31 & 53 & 14 \\
40 & 28 & 3 & 82 & 45 & 14 \\
42 & 3 & 34 & 14 & 69 & 22 \\
42 & 19 & 82 & 14 & 58 & 25 \\
42 & 23 & 25 & 14 & 62 & 83 \\
42 & 25 & 3 & 14 & 74 & 37 \\
42 & 77 & 46 & 23 & 10 & 14 \\
44 & 13 & 21 & 14 & 77 & 46 \\
44 & 14 & 21 & 77 & 13 & 46 \\
44 & 25 & 83 & 7 & 48 & 17 \\
44 & 28 & 5 & 54 & 7 & 17 \\
44 & 57 & 7 & 28 & 19 & 6 \\
46 & 13 & 77 & 21 & 14 & 44 \\
46 & 28 & 7 & 44 & 0 & 13 \\ \hline
\end{tabular}}
\quad
\footnotesize{\begin{tabular}{|cccccc|}
\hline
46 & 77 & 14 & 10 & 23 & 42 \\
46 & 83 & 77 & 21 & 33 & 23 \\
47 & 0 & 80 & 28 & 34 & 17 \\
47 & 14 & 21 & 45 & 78 & 14 \\
47 & 34 & 3 & 31 & 80 & 14 \\
47 & 37 & 14 & 13 & 69 & 16 \\
47 & 80 & 31 & 3 & 14 & 34 \\
48 & 17 & 69 & 13 & 14 & 36 \\
48 & 83 & 28 & 76 & 25 & 23 \\
50 & 14 & 31 & 13 & 14 & 62 \\
50 & 17 & 14 & 48 & 83 & 22 \\
50 & 31 & 78 & 7 & 33 & 14 \\
51 & 0 & 17 & 28 & 39 & 4 \\
51 & 14 & 39 & 12 & 77 & 24 \\
51 & 17 & 0 & 39 & 4 & 28 \\
51 & 20 & 0 & 33 & 82 & 28 \\
51 & 39 & 77 & 23 & 10 & 17 \\
51 & 82 & 19 & 47 & 14 & 20 \\
52 & 33 & 28 & 83 & 0 & 21 \\
53 & 0 & 32 & 83 & 28 & 22 \\
53 & 7 & 26 & 14 & 17 & 54 \\
53 & 14 & 72 & 36 & 17 & 19 \\
53 & 17 & 14 & 51 & 25 & 4 \\
53 & 17 & 19 & 25 & 58 & 7 \\
53 & 22 & 69 & 32 & 13 & 14 \\
53 & 26 & 25 & 28 & 77 & 3 \\
54 & 14 & 17 & 26 & 53 & 7 \\
54 & 14 & 81 & 12 & 39 & 28 \\
54 & 23 & 28 & 83 & 1 & 32 \\
54 & 28 & 1 & 32 & 83 & 23 \\
58 & 7 & 25 & 78 & 28 & 30 \\
58 & 14 & 25 & 53 & 8 & 21 \\
58 & 14 & 41 & 28 & 83 & 16 \\
58 & 21 & 25 & 53 & 14 & 8 \\
58 & 25 & 42 & 14 & 82 & 19 \\
58 & 28 & 39 & 5 & 3 & 22 \\
58 & 39 & 12 & 14 & 2 & 33 \\
60 & 7 & 44 & 13 & 32 & 14 \\
60 & 28 & 23 & 32 & 10 & 77 \\
62 & 13 & 6 & 39 & 37 & 14 \\
62 & 13 & 39 & 37 & 6 & 14 \\
62 & 14 & 83 & 20 & 41 & 28 \\
64 & 77 & 31 & 14 & 21 & 33 \\
66 & 14 & 49 & 28 & 7 & 24 \\ \hline
\end{tabular}}
\quad
\footnotesize{\begin{tabular}{|cccccc|}
\hline
66 & 28 & 29 & 17 & 68 & 19 \\
66 & 33 & 83 & 18 & 39 & 17 \\
68 & 14 & 51 & 23 & 10 & 31 \\
68 & 15 & 78 & 28 & 37 & 25 \\
68 & 19 & 17 & 37 & 83 & 36 \\
68 & 31 & 0 & 37 & 10 & 28 \\
69 & 7 & 24 & 52 & 28 & 19 \\
69 & 8 & 34 & 3 & 36 & 28 \\
69 & 23 & 16 & 28 & 72 & 33 \\
69 & 28 & 48 & 22 & 3 & 20 \\
69 & 42 & 8 & 14 & 39 & 18 \\
69 & 48 & 31 & 24 & 7 & 14 \\
72 & 19 & 44 & 17 & 39 & 7 \\
72 & 28 & 33 & 16 & 69 & 23 \\
72 & 39 & 19 & 76 & 28 & 23 \\
74 & 14 & 21 & 54 & 37 & 17 \\
74 & 37 & 14 & 25 & 42 & 3 \\
74 & 51 & 18 & 17 & 39 & 14 \\
76 & 23 & 39 & 19 & 72 & 28 \\
76 & 37 & 28 & 0 & 39 & 13 \\
77 & 3 & 42 & 39 & 14 & 26 \\
77 & 17 & 80 & 39 & 37 & 24 \\
77 & 19 & 64 & 28 & 36 & 17 \\
77 & 24 & 51 & 14 & 39 & 12 \\
77 & 28 & 24 & 69 & 37 & 16 \\
77 & 32 & 42 & 14 & 83 & 30 \\
77 & 39 & 12 & 40 & 3 & 28 \\
78 & 25 & 28 & 37 & 15 & 68 \\
80 & 14 & 31 & 3 & 34 & 47 \\
80 & 17 & 39 & 63 & 28 & 18 \\
80 & 23 & 46 & 13 & 47 & 14 \\
80 & 43 & 17 & 28 & 77 & 30 \\
82 & 14 & 19 & 47 & 51 & 20 \\
82 & 28 & 25 & 44 & 72 & 19 \\
82 & 33 & 72 & 14 & 39 & 26 \\
82 & 37 & 19 & 47 & 83 & 25 \\
83 & 7 & 33 & 32 & 23 & 60 \\
83 & 16 & 28 & 44 & 77 & 32 \\
83 & 21 & 14 & 47 & 52 & 19 \\
83 & 22 & 48 & 14 & 50 & 17 \\
83 & 28 & 22 & 41 & 64 & 14 \\
83 & 30 & 14 & 32 & 77 & 42 \\
83 & 33 & 14 & 66 & 38 & 21 \\
83 & 39 & 18 & 66 & 33 & 17 \\ \hline
\end{tabular}}

\end{document}